\documentclass[11pt]{article}

\usepackage{amsfonts,amsmath,latexsym,color,epsfig,hyperref, amsthm}
\newtheorem{theorem}{Theorem}

\newtheorem{conjecture}{Conjecture}

\input{epsf}

\makeatletter
\def\Ddots{\mathinner{\mkern1mu\raise\p@
\vbox{\kern7\p@\hbox{.}}\mkern2mu
\raise4\p@\hbox{.}\mkern2mu\raise7\p@\hbox{.}\mkern1mu}}
\makeatother

\title{Extremal numbers of cycles revisited}
\author{David Conlon\thanks{Department of Mathematics, California Institute of Technology, Pasadena, CA 91125. Email: {\tt dconlon@caltech.edu}.}}
\date{}

\begin{document}
\maketitle

\begin{abstract}
We give a simple geometric interpretation of an algebraic construction of Wenger that yields $n$-vertex graphs with no cycle of length $4$, $6$ or $10$ and close to the maximum number of edges.
\end{abstract}

What is the maximum number of edges in an $n$-vertex graph containing no cycle $C_{2k}$ of length $2k$? This is one of the most closely studied yet elusive problems in combinatorics. 
Despite many decades of intense interest in this and related problems in extremal graph theory~\cite{Ver16}, the answer is only reasonably well understood for cycles of length $4$, $6$ and $10$. 

If we write $\textrm{ex}(n, H)$ for the maximum number of edges in an $n$-vertex graph containing no copy of the graph $H$, then a result attributed to Erd\H{o}s~\cite{BS74, Erd65} says that $\textrm{ex}(n, C_{2k}) \leq C n^{1 + 1/k}$ for some constant $C$ depending only on $k$. 
For $C_4$, a matching lower bound, due to Klein, already appeared in a paper of Erd\H{o}s~\cite{Erd38} from 1938. For $C_6$ and $C_{10}$, constructions matching the upper bound were found by Benson~\cite{Ben66} and Singleton~\cite{Sin66} in 1966, though several alternative constructions have since been found~\cite{LU95, MM05, Wen91}. None of these constructions can be described as simple, though that of Wenger~\cite{Wen91} has the best claim. Here we show that this claim is justified, by rephrasing his algebraic construction in geometric terms that make its properties manifest.

Let $q$ be a prime power and $\mathbb{F}_q$ the finite field of order $q$. For each $x \in \mathbb{F}_q^k$ and $z \in \mathbb{F}_q$, we form the line 
\[\{x + y \cdot (1, z, z^2, \dots, z^{k-1}) : y \in \mathbb{F}_q\}.\]
The number of distinct lines of this form is exactly $q^k$, since there are $q$ distinct directions, determined by the different values for $z$, and exactly $q^{k-1}$ parallel lines in each direction. We form a bipartite graph $D_k(q)$ between two sets $P$ and $L$ of order $q^k$, where the vertices of $P$ are indexed by the points of $\mathbb{F}_q^k$, the vertices of $L$ are indexed by the $q^k$ lines described above and there is an edge between $p \in P$ and $\ell \in L$ if and only if $p \in \ell$. Our main result is as follows.

\begin{theorem} \label{thm:main}
For $k = 2$, $3$ and $5$, $D_k(q)$ is a $C_{2k}$-free bipartite graph between sets $P$ and $L$ of order $n = q^k$ with $n^{1+1/k}$ edges.
\end{theorem}

\begin{proof}
The bound on the number of edges follows from the fact that each of the $q^k$ lines contains exactly $q$ points, so there are $q^{k + 1} = n^{1 + 1/k}$ edges in total. Note now that any cycle in $D_k(q)$ is of the form $p_1 \ell_1 p_2 \ell_2 \dots p_t \ell_t p_1$, where $p_i \in P$ and $\ell_i \in L$ for all $1 \leq i \leq t$. We make two simple observations about these cycles:

\begin{itemize}

\item
For any $1 \leq i \leq t$ (taken mod $t$), $\ell_i$ and $\ell_{i+1}$ cannot be parallel. 

\item
If $t \leq k$, then, for any $1 \leq i \leq t$, there must be another line $\ell_{i'}$ with $1 \leq i' \leq t$ parallel to $\ell_i$.

\end{itemize}

The first observation is obvious, since two parallel lines, both passing through the point $p_{i+1}$, must coincide. For the second observation, note that if $\ell_i$ has direction determined by $z_i$ for each $1 \leq i \leq t$, then the difference $p_{i+1} - p_i$ is a non-zero multiple of $(1, z_i, z_i^2, \dots, z_i^{k-1})$. Adding over all $i$, we have that
\[0 = \sum_{i = 1}^t (p_{i+1} - p_i) = \sum_{i=1}^t a_i (1, z_i, z_i^2, \dots, z_i^{k-1})\]
for some collection of non-zero coefficients $a_i$. But it is a well-known fact, proved by considering the Vandermonde determinant, that any $k$ distinct vectors of the form $(1, z, z^2, \dots, z^{k-1})$ are linearly independent. Hence, for the sum $\sum_{i=1}^t a_i (1, z_i, z_i^2, \dots, z_i^{k-1})$ to be zero, each direction must appear at least twice. That is, every line $\ell_i$ has at least one parallel line $\ell_{i'}$. We can now dispense with each case as a bulletpoint:

\begin{itemize}

\item
No $D_k(q)$ with $k \geq 2$ contains a cycle of length $4$, since any such cycle $p_1 \ell_1 p_2 \ell_2 p_1$ must be such that $\ell_1$ and $\ell_2$ are both parallel and not parallel by the two observations above.

\item
No $D_k(q)$ with $k \geq 3$ contains a cycle of length $6$, since any such cycle $p_1 \ell_1 p_2 \ell_2 p_3 \ell_3 p_1$ must again be such that the three  lines $\ell_1, \ell_2$ and $\ell_3$ are all parallel and all not parallel by our two observations.

\item
No $D_k(q)$ with $k \geq 5$ contains a cycle of length $10$. Indeed, any such cycle $p_1 \ell_1 p_2 \ell_2 p_3 \ell_3 p_4 \ell_4 p_5 \ell_5 p_1$ must have one group of two parallel lines and another group of three parallel lines. But then it is impossible to enforce the condition that $\ell_i$ not be parallel to $\ell_{i+1}$ for all $1 \leq i \leq 5$ (taken mod $5$). \qedhere
\end{itemize} 
\end{proof}

Suppose now that $\theta_{k, \ell}$ is the graph consisting of $\ell$ internally disjoint paths of length $k$, each with the same endpoints. In particular, $\theta_{k,2}$ is just $C_{2k}$, so the problem of determining $\textrm{ex}(n, \theta_{k, \ell})$ extends the problem of determining $\textrm{ex}(n, C_{2k})$. The extremal numbers for these theta graphs were first studied by Faudree and Simonovits \cite{FS83}, who proved that $\textrm{ex}(n, \theta_{k, \ell}) \leq C n^{1 + 1/k}$ for a constant $C$ depending only on $k$ and $\ell$. On the other hand, a result of Conlon~\cite{Con19} says that for any $k$ there exists a natural number $\ell$ and a positive constant $c$ such that $\mathrm{ex}(n, \theta_{k, \ell}) \geq c n^{1 + 1/k}$.

The following question therefore becomes valid: for each $k$, what is the smallest $\ell$ such that $\mathrm{ex}(n, \theta_{k, \ell}) \geq c n^{1 + 1/k}$ for some positive $c$? The results on $\textrm{ex}(n, C_{2k})$ show that for $k = 2$, $3$ or $5$ we can take $\ell = 2$. A much more recent result of Verstra\"ete and Williford \cite{VW19} says that for $k = 4$ it suffices to take $\ell = 3$. As a bonus, we show that the graph $D_4(q)$ has a one-sided version of the property satisfied by the Verstra\"ete--Williford graph, thus giving a natural interpolation between the cases $k = 3$ and $5$.

\begin{theorem}
$D_4(q)$ is a bipartite graph between sets $P$ and $L$ of order $n = q^4$ with $n^{5/4}$ edges such that any two vertices in $P$ have at most $2$ paths of length $4$ between them.
\end{theorem}

\begin{proof}
A theta graph $\theta_{4, \ell}$ whose endpoints are $p, p' \in P$  consists of paths $p \ell_{1,j} p_{2, j} \ell_{2, j} p'$ with $1 \leq j \leq \ell$. However, any two such paths yield a cycle, so we can still apply our basic observations:

\begin{itemize}

\item
No $D_k(q)$ with $k \geq 4$ contains a theta graph $\theta_{4,3}$ with both endpoints $p, p' \in P$. Indeed, any such graph consists of paths $p \ell_{1,j} p_{2,j} \ell_{2,j} p'$ for $1 \leq j \leq 3$ where $\ell_{1, j}$ and $\ell_{2,j}$ are not parallel, by our first observation. But, since $p \ell_{1,1} p_{2,1} \ell_{2,1} p' \ell_{2,2} p_{2,2} \ell_{1,2} p$ is a cycle, we must have, by our second observation, that $\ell_{1,1}$ is parallel to $\ell_{2,2}$ and $\ell_{2,1}$ is parallel to $\ell_{1,2}$. Similarly, considering the cycle $p \ell_{1,1} p_{2,1} \ell_{2,1} p' \ell_{2,3} p_{2,3} \ell_{1,3} p$, we must have that $\ell_{1,1}$ is parallel to $\ell_{2,3}$ and $\ell_{2,1}$ is parallel to $\ell_{1,3}$. But then the cycle $p \ell_{1,2} p_{2,2} \ell_{2,2} p' \ell_{2,3} p_{2,3} \ell_{1,3} p$ violates the first observation, since $\ell_{2,2}$ and $\ell_{2,3}$ are parallel. \qedhere

\end{itemize}
\end{proof}

We conclude with a conjecture. It is now quite commonplace to believe that the true value of certain extremal numbers lie below what the classical arguments give. We suspect that this should already be the case for $C_8$, that is, that $\textrm{ex}(n, C_8) = o(n^{5/4})$. Proving this is likely to be exceptionally difficult, but a first step might be to show that no construction of the type studied in this paper can yield  the lower bound. Concretely, we have the following conjecture.  The definition of a line in $\mathbb{F}_q^k$ agrees with that used throughout, namely, a set of the form $\{x + y \cdot z: y \in \mathbb{F}_q\}$ for some $x, z \in \mathbb{F}_q^k$.

\begin{conjecture} 
The maximum number of lines in $\mathbb{F}_q^4$ containing no $C_4$ of lines, that is, four distinct lines $\ell_1, \ell_2, \ell_3, \ell_4$ such that $\ell_i$ and $\ell_{i+1}$ intersect in distinct points for all $1 \leq i \leq 4$ (taken mod $4$), is $o(q^4)$.
\end{conjecture}

A similar conjecture can be made for any cycle of length at least $6$.

\vspace{3mm}
{\bf Note added.} After this article was accepted, Tibor Szab\'o drew my attention to his lecture notes~\cite{Sz18}, where Wenger's construction is also described in very similar terms.

\end{document}